\RequirePackage{fix-cm}
\documentclass[smallextended]{svjour3}       
\smartqed  
\usepackage{graphicx}

\usepackage[english]{babel}
\usepackage[T1]{fontenc}
\usepackage[latin1]{inputenc}

\usepackage{amsmath}
\DeclareMathOperator*{\argmax}{arg\,max}
\DeclareMathOperator*{\argmin}{arg\,min}

\usepackage{amsfonts,amssymb,amsmath}
\usepackage{bbm}
\usepackage{xcolor}
\usepackage{graphics}
\usepackage{graphicx}
\usepackage{mathtools}
\usepackage{soul}

\newcommand{\ba}{\vspace{0mm}\begin{equation*}\begin{aligned}}
\newcommand{\ea}{\vspace{0mm}\end{aligned}\end{equation*}}
\newcommand{\ban}{\vspace{0mm}\begin{equation}\begin{aligned}}
\newcommand{\ean}{\vspace{0mm}\end{aligned}\end{equation}}

\newcommand{\ben}{\vspace{0mm}\begin{equation}}
\newcommand{\een}{\vspace{0mm}\end{equation}}

\begin{document}

\title{Pedigree in the biparental Moran model}

\author{Camille Coron \and Yves Le Jan}
\institute{Camille Coron \at Universit\'e Paris-Saclay, CNRS, Laboratoire de math\'ematiques d'Orsay, 91405, Orsay, France  \email{camille.coron@universite-paris-saclay.fr}  \and Yves Le Jan \at NYU, Universit\'e Paris-Saclay, CNRS, Laboratoire de math\'ematiques d'Orsay, 91405, Orsay, France}
 
\maketitle 
 
\begin{abstract}
Our goal is to study the genetic composition of a population in which each individual has $2$ parents, who contribute equally to the genome of their offspring. We use a biparental Moran model, which is characterized by its fixed number $N$ of individuals. We fix an individual and consider the proportions of the genomes of all individuals living $n$ time steps later, that come from this individual. When $n$ goes to infinity, these proportions all converge almost surely towards the same random variable.  When $N$ then goes to infinity, this random variable multiplied by $N$ (i.e. the stationary weight of any ancestor in the whole population) converges in law towards the mixture of a Dirac measure in 0 and an exponential law with parameter $1/2$, and the weights of several given ancestors are independent. This gives an explicit formula for the limiting (deterministic) distribution of all ancestors' weights.
\end{abstract}



\section{Introduction}

\smallskip
This paper deals with the role of the pedigree, i.e. the complete description of all ancestral lines, in the genetic composition of a population with biparental genetic transmission.

In population genetics models with monoparental transmission the number of ancestors of a sample of individuals classically decreases backwards in time, till the occurrence of a common ancestor from which all the genetic material of the sampled individuals necessarily derives. \textcolor{black}{On the contrary}, as was already observed in several articles, notably \cite{Derrida2000,Chang1999}, we expect the genetic material of a sample of individuals in populations with biparental genetic transmission to originate from a large number of ancestors, that is not decreasing nor constant backwards in time. From this observation, an element of interest in these populations, is the contribution of a set of ancestors to the genetic material of a sample of individuals. As mentioned in several articles, (\cite{Wakeley2005,Wakeley2012,Wakeley2016,Wilton2017}), this contribution is highly dependent on the pedigree, i.e. the \textcolor{black}{joint succession of ancestors}, of the sampled individuals.

Biparental genealogies have received some interest, notably in \cite{Chang1999,Derrida2000,GravelSteel2015}, in which time to recent common ancestors and ancestors' weights are investigated for the Wright-Fisher biparental model. In \cite{Derrida2000}, a convergence result was stated based on a first and second moment calculation and an unproved independence ansatz.  The time to \textcolor{black}{the} recent common ancestors for the biparental Moran model is also studied in \cite{Linder2009CommonAI}. \textcolor{black}{The convergence of the genealogical process of a sample of individuals} towards a diploid coalescent is considered notably in \cite{Mohle1998,MohleSagitov2003,Birkneretal2013,BirknerLiuSturm2018} respectively for the diploid Wright-Fisher model and for a $2$-sex Wright-Fisher model. \cite{MatsenEvans2008} and \cite{BartonEtheridge2011} study the link between pedigree, individual reproductive success and genetic contribution. Finally, in \cite{BairdBartonEtheridge2003,Lambertetal2018}, recombination between loci and its impact on genome transmission in populations with sexual reproduction are investigated.

\smallskip
In this article we consider \textcolor{black}{the Moran biparental model in which we assume that both parents and the replaced individual are in different sites} and study the amounts of genetic material transmitted by all ancestors in future generations. More precisely, we define the "weight" of an ancestor in a given sampled individual as the probability that a given gene of the sampled individual comes from this ancestor. Note that this definition implicitely conceptualizes the genome as a concatenation of infinitely-many independently segregating loci that have no impact on reproductive success. This idealized model however has the advantage of providing explicit formulas for the asymptotic weights of ancestors. We prove (Theorem \ref{mainthm} and Corollary \ref{maincor}) first that the weights of an ancestor in all individuals are asymptotically equal \textcolor{black}{ (this was proved in \cite{Chang1999} for the Wright-Fisher model and one can think it holds under quite general assumptions)}. Secondly, we prove that the total weights of $l$ ancestors in the population \textcolor{black}{(properly rescaled by a factor of $N$)} converge in law when the number of individuals $N$ goes to infinity, to a vector of $l$ independent random variables, that are equal to 0 with probability $1/2$, or follow an exponential law with parameter $1/2$. This result also gives \textcolor{black}{(Corollary \ref{corhisto})} that the properly rescaled plot of ordered weights of all ancestors converges to the inverse of the cumulative distribution function associated to this limiting law, which is illustrated by simulation outcomes \textcolor{black}{(Figure \ref{FigBarplot})}. To prove these results we first introduce the pedigree graph of \textcolor{black}{the population}, representing the parental \textcolor{black}{links} between individuals. The genealogy of a gene is then a random walk on this graph, going backwards in time. The moments of the weights of $l$ ancestors are then studied by considering the dynamics and the stationary distribution of a $k$-particle random walk on the random pedigree graph (for any $k\geq l$). We next find an appropriate projection of this $k$-particles random walk (on the space of multiplicities of multisets with cardinality $k$) which remains a Markov chain due to symmetries of the model. The particular dynamics of this new Markov chain when the number $N$ of individuals is large finally \textcolor{black}{leads us to use} to the characterization of \textcolor{black}{the stationary distribution of Markov chains} based on oriented spanning trees, given notably in \cite{Shubert1975}. The limiting (when $N$ goes to infinity) stationary distribution of the $k$ particle Markov chain as well as the limiting law of ancestors\textcolor{black}{'} weights are then derived.

\section{Biparental Moran model}\label{sec:model}

We consider a population of $N$ individuals following a neutral biparental Moran model. More precisely, individuals are numbered by $\{1, 2, ..., N\}=I$ and at each discrete time step $t\in\mathbb{N}=\{0,1,2,...\}$, a triplet of distinct individuals $(\pi_t,\mu_t,\kappa_t)\in I^3$ is chosen uniformly at random among the population. The first two individuals, $\pi_t$ (father) and $\mu_t$ (mother), produce one new offspring that replaces the individual $\kappa_t\in I$. This forms the population at time $t+1$. Note that one could alternatively choose the three individuals $\pi_t$, $\mu_t$ and $\kappa_t$ uniformly and independently at random in the population, at any time $t$. The difference between these two models should be negligible when $N$ is large and all the results stated from now should remain true for this second model, though some calculations are simpler in the first model that we now consider. 

This reproduction dynamics defines an oriented random graph on $I\times\mathbb{N}$ (as represented in Figure \ref{FigPedigree}), denoted $\mathcal{G}_N$, representing the pedigree of the population, such that between time $t+1$ and time $t$, two arrows are drawn from $(\kappa_t,t+1)$ to $(\pi_t,t)$ and $(\mu_t,t)$ respectively and $N-1$ arrows are drawn from $(x,t+1)$ to $(x,t)$ for each $x\in I\setminus\{\kappa_t\}$.

\begin{figure}
\begin{center}
\includegraphics[scale=0.3, trim=5cm 4cm 3cm 0cm]{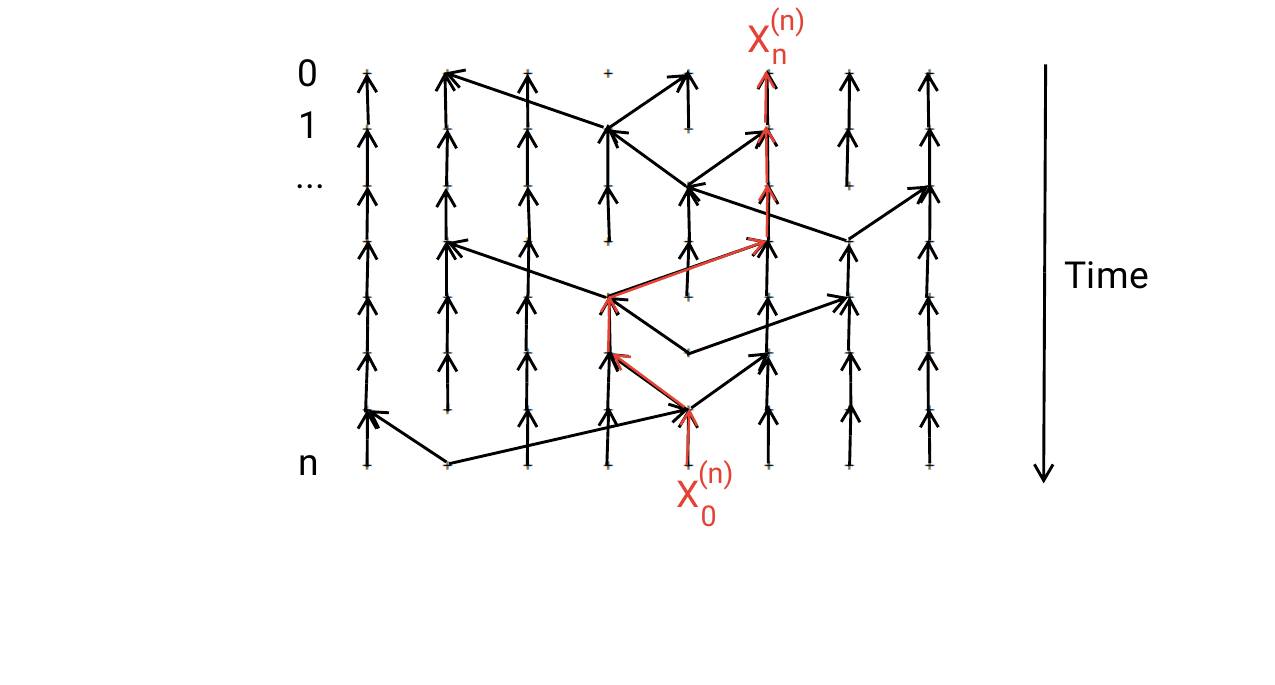}
\end{center} 
\caption{Graph representing the pedigree of a population with $8$ individuals, during $7$ time steps. \textcolor{black}{The time orientation is from past to future.}\label{FigPedigree}}
\end{figure}

Now let us consider a gene (portion of genome) of an individual $i$ present in the population at time $n$. The genealogy of this gene (i.e. the individual in which a copy of this gene was present, assuming no mutation and no recombination, at each time $t=n-k\leq n$), denoted by $(X^{(n)}_k, n-k)_{0\leq k\leq n}$, is a random walk on this random graph, starting from the position $(i,n)$. One can then consider the random variable  

\begin{equation} \label{eq:defA} A_n(i,j)=\mathbb{P}(X^{(n)}_n=j|X^{(n)}_0=i,\mathcal{G}_N)\end{equation}
(which is a deterministic function of the random graph $\mathcal{G}_N$). This quantity, given the genealogy $\mathcal{G}_N$, is the probability that any gene of individual $i$ living \textcolor{black}{in} generation $n$ comes from ancestor $j$ living at generation $0$. If genome size is very large and \textcolor{black}{the evolutions of distant genes} are sufficiently decorrelated, we can expect this quantity to be close to the proportion of genes of individual $i$ that come from individual $j$. This quantity will also be called the weight of the ancestor $j$  in the genome of individual $i$. It is also natural to consider the random variable
$$M_n(j)=\sum_{i=1}^N A_n(i,j)$$ that measures the weight of the ancestor $j$, in the population living $n$ time steps later. Note that $M_n(j)\in[0,N]$ and that $\sum_{j=1}^NM_n(j)=N$ for all $n$.

\bigskip
Let us denote by $\mathcal{F}_n$ the filtration associated to the restriction of $\mathcal{G}_N$ to $I\times\{0,...,n\}$. We \textcolor{black}{then} have the following lemma:

\begin{lemma}\label{lemmaM}
For all $j$ the stochastic process $(M_n(j))_{n\in\mathbb{N}}$ is a martingale with respect to the filtration $(\mathcal{F}_n)_{n\in\mathbb{Z}_+}$, whose law is independent of $j$.
\end{lemma}

\begin{proof} 
\begin{align*}M_{n+1}(j)&=\sum_{i=1}^N \mathbb{P}(X_{n+1}^{(n+1)}=j|X_{0}^{(n+1)}=i,\mathcal{G}_N)\quad\quad\text{by definition}\\&=\sum_{i=1}^N\sum_{i'=1}^N \mathbb{P}(X_{n+1}^{(n+1)}=j,X_{1}^{(n+1)}=i'|X_{0}^{(n+1)}=i,\mathcal{G}_N)\\&=\sum_{i=1}^N\sum_{i'=1}^N \mathbb{P}(X_{n+1}^{(n+1)}=j|X_{1}^{(n+1)}=i',X_{0}^{(n+1)}=i,\mathcal{G}_N)\\&\hspace{1cm}\times\mathbb{P}(X_{1}^{(n+1)}=i'|X_{0}^{(n+1)}=i,\mathcal{G}_N)\\&=\sum_{i=1}^N\sum_{i'=1}^N \mathbb{P}(X_{n+1}^{(n+1)}=j|X_{1}^{(n+1)}=i',\mathcal{G}_N)\\&\hspace{1cm}\times\mathbb{P}(X_{1}^{(n+1)}=i'|X_{0}^{(n+1)}=i,\mathcal{G}_N) \quad\text{by the Markov property.}\\&=\sum_{i=1}^N\sum_{i'=1}^N \mathbb{P}(X_{n}^{(n)}=j|X_{0}^{(n)}=i',\mathcal{G}_N)\mathbb{P}(X_{1}^{(n+1)}=i'|X_{0}^{(n+1)}=i,\mathcal{G}_N)\\&=\sum_{i=1}^N\sum_{i'=1}^N A_n(i',j)\mathbb{P}(X_{1}^{(n+1)}=i'|X_{0}^{(n+1)}=i,\mathcal{G}_N)\quad\quad\text{by definition of $A_n$.}\end{align*}

Therefore since $A_n(i',j)\in\mathcal{F}_n$ and $\mathbb{P}(X_{1}^{(n+1)}=i'|X_{0}^{(n+1)}=i,\mathcal{G}_N)$ is independent of $\mathcal{F}_n$,
\begin{align*}\mathbb{E}(M_{n+1}(j)|\mathcal{F}_n)&=\sum_{i'=1}^NA_n(i',j)\sum_{i=1}^N \mathbb{E}\left(\mathbb{P}(X_{1}^{(n+1)}=i'|X_{0}^{(n+1)}=i,\mathcal{G}_N)\right)\\&=\sum_{i'=1}^NA_n(i',j)N\mathbb{P}(X_1^{(n+1)}=i') \\&=\sum_{i'=1}^NA_n(i',j) \quad\text{by exchangeability of elements of $I$}\\&=M_n(j).\end{align*} \qed
 \end{proof}

As $M_n(j)\in[0,N]$ for all $n$ and $j$, the martingale $(M_n(j))_{n\in\mathbb{N}}$ is bounded and therefore converges almost surely to some random variable $M_{\infty}(j)$ when $n$ goes to infinity. 
Our main result is now the following

\begin{theorem}\label{mainthm}
\begin{description}
\item[$(i)$] For all $j\in I$, there exists a random variable $A_{\infty}(j)$ such that  $$A_n(i,j)\xrightarrow[n\rightarrow\infty]{} A_{\infty}(j) \quad a.s.$$ In particular,

$$M_{\infty}(j)=N A_{\infty}(j) \quad a.s.$$

Note that the distribution of $M_{\infty}$ depends on $N$.

\item[$(ii)$] For any $l\leq N$ and any $k_1,...,k_l \in \mathbb{Z}_+$, 

\begin{equation}\label{eq:limitlawM}\mathbb{E}\left(M_{\infty}^{k_1}(1)...M_{\infty}^{k_l}(l)\right) \xrightarrow[N\rightarrow\infty]{} \prod_{i=1}^l 2^{k_i-1}k_i!\quad.\end{equation}
\end{description}
\end{theorem}

\bigskip
A corollary of the second point of this theorem is

\begin{corollary} \label{maincor} For $l\leq N$, let us define $R_1,...,R_l$, independent random variables on $\mathbb{R}_+$, such that $R_i$ is equal to $0$ with probability $1/2$ or follows an exponential law with parameter $1/2$). \textcolor{black}{Then}
$$\left(M_{\infty}(1),...,M_{\infty}(l)\right) \xRightarrow[N\rightarrow\infty]{} (R_1,...,R_l).$$
\end{corollary}

This theorem also gives an explicit formula for the limiting distribution of all ancestors' weights.

\begin{corollary}\label{corhisto}
For any $i\in I$, let  $\tilde{M}_i$ be the $i$-th order statistic of the vector $(M_{\infty}(1),...,M_{\infty}(N))$. Then  for all $u\in\textcolor{black}{[}0,1\textcolor{black}{)}$, when $N$ goes to infinity,
$$\tilde{M}_{[Nu]}\rightarrow - 2 \ln(2(1-u))\mathbf{1}_{\textcolor{black}{\{}u>1/2\textcolor{black}{\}}} \quad\quad \textcolor{black}{\text{in probability}}$$ which coincides with the quantile function of the random variables $R_i$. 
\end{corollary}
This convergence is illustrated in Figure \ref{FigBarplot}.

\begin{figure}
\begin{center}
\includegraphics[scale=1]{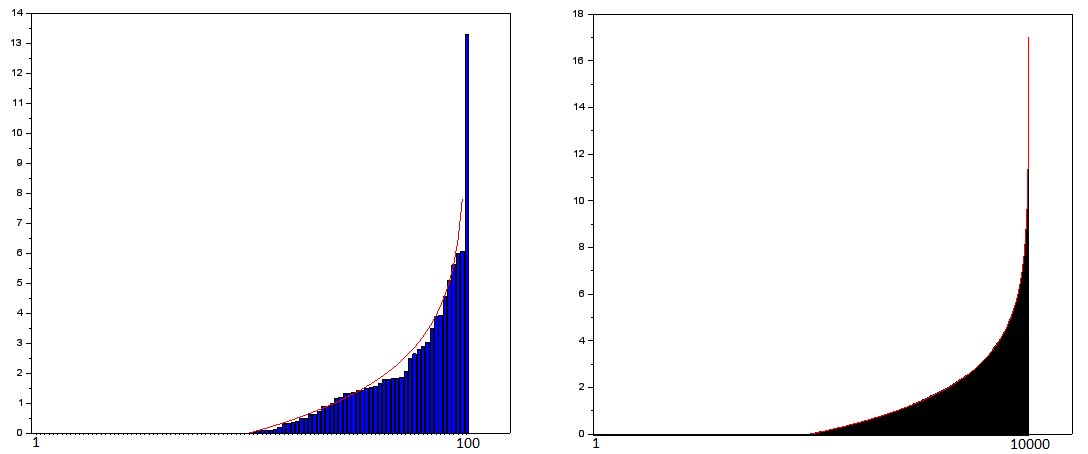}
\end{center} 
\caption{We run one simulation of the Moran model for $N=100$ (left) and $N=10000$ (right): in blue/black, we plot ancestors' weights \textcolor{black}{ranked in increasing order} after $100000$ time steps, whereas in red is the function $x\mapsto -\mathbf{1}_{x>N/2}\times 2\ln(2(1-x/N))$.\label{FigBarplot}}
\end{figure}

\begin{remark}
\textcolor{black}{Our main result consists in giving the asymptotic distribution of ancestors' weights. The proof of this result relies on guessing a priori what will be the form of the limiting distribution. This guess (from which Formula \eqref{eq:limitlawM} follows) can be done by assuming asymptotic independence of ancestors' weights. Under this hypothesis, we} find that the Laplace transform of the asymptotic distribution of $M_{\infty}(1)$ should satisfy the equation $h(\lambda)=\mathbb{E}(\exp(-\lambda M_{\infty}(1)))=\frac{1}{3}+\frac{2}{3}h\left(\frac{\lambda}{2}\right)h(\lambda)$, which leads to conjecture the limiting distribution of the weight of a given ancestor, as done in \cite{Derrida2000} for the Wright-Fisher model \textcolor{black}{(Equation (A.13))}. 
\end{remark}

\begin{remark} Similar results should hold when considering a non biological extension of our model in which the genealogical graph has constant reverse branching multiplicity $m$ (here $m=2$), i.e. each "child" has $m$ "parents". In that case the limiting weight of a given ancestor is expected to be equal to $0$ with probability $1/m$, or (with probability $1-1/m$) to follow an exponential law with parameter $1/m$.\end{remark}

\section{Proofs}

\begin{proof}[Proof of Theorem \ref{mainthm}]
$(i)$ The dynamics of $A_n$ defined in Equation \eqref{eq:defA} is the following: 

$$\left\{\begin{tabular}{l}
$A_{n+1}(i,j)=A_n(i,j) \quad\quad \text{if $i\neq \kappa_n$}$ \\ 
$A_{n+1}(\kappa_n,j)=\frac{1}{2}\left(A_n(\pi_n,j)+A_n(\mu_n,j)\right)$ \\ 
\end{tabular}  \right.$$
Therefore if we define $L_n(j)=\max_{i\in I} A_n(i,j)$ and $l_n(j)=\min_{i\in I} A_n(i,j)$ for all $n\in\mathbb{Z_+}$ and $j\in I$, the sequences $(L_n(j))_n$ and $(l_n(j))_n$ are respectively \textcolor{black}{non-increasing and non-decreasing} in $n$, for all $j\in I$. Then the difference $L_{n}(j)-l_{n}(j)$ is decreasing in $n$.  
Now at each time step $n$, let us define \textcolor{black}{$i_{min}(j)$ any element in $\argmin_i A_n(i,j)$, and $i_{max}(j)$ any element in $\argmax_i A_n(i,j)$}, and $$D_n=\{\text{at time $n, i_{\min}(j)$ is killed and replaced by an offspring of $i_{\max}(j)$}\}.$$ We know that $\mathbb{P}(D_n|\mathcal{F}_{n-1})=1/N^2$ so by the conditional Borel-Cantelli lemma \textcolor{black}{(\cite{Levy1937})} infinitely many events  $\{D_n,n\in\mathbb{N}\}$ occur. Note that once at most $N-1$ such events have occurred, the difference $L_{n}(j)-l_{n}(j)$ is at least divided by $2$. Therefore this difference $L_{n}(j)-l_{n}(j)$ goes almost surely to $0$ when $n$ goes to infinity.

%

Therefore $A_n(i,j)$ converges to $A_{\infty}(j)=L_{\infty}(j)$ almost surely, when $n$ goes to infinity. This also implies that  $M_{\infty}(j)=\sum_{i=1}^N A_{\infty}(j)=N A_{\infty}(j)$.

\medskip
$(ii)$ Note that $$M_n(j)=\sum_{i=1}^N A_n(i,j)=N\mathbb{P}\left(X_n^{(n)}=j|X_0^{(n)}\sim\mathcal{U}(I),\mathcal{G}_N\right)$$ where $\mathcal{U}(I)$ denotes the uniform law on $I$. For each $j\in I$, $M_n(j)$ is a deterministic function of the random pedigree graph $\mathcal{G}_N$, and the quantities $M_n(1)$, $M_n(2)$, ..., $M_n(N)$ are random variables adding up to $N$ that all depend on the same random graph $\mathcal{G}_N$. Let $(X^{(1,n)}_i,n-i)_{i\leq n}$, $(X^{(2,n)}_i,n-i)_{i\leq n}$, ..., $(X^{(k,n)}_i,n-i)_{i\leq n}$ be $k$ independent random walks on $\mathcal{G}_N$ starting at generation $n$, such that for all $j$, the positions $X^{(j,n)}_0$ are independent and follow the uniform law on $I$. For each $k_1, k_2,...,k_l\in\mathbb{Z}_+$ such that $\sum_{j=1}^lk_j=k$, we have

\begin{align*}M_n^{k_1}(1)...M_n^{k_l}(l)&=N^k\mathbb{P}\left(X_n^{(1,n)}=1|\mathcal{G}_N\right)\mathbb{P}\left(X_n^{(2,n)}=1|\mathcal{G}_N\right)...\,\mathbb{P}\left(X_n^{(k_1,n)}=1|\mathcal{G}_N\right)\\&\times\mathbb{P}\left(X_n^{(k_1+1,n)}=2|\mathcal{G}_N\right)...\mathbb{P}\left(X_n^{(k_1+...+k_l,n)}=l|\mathcal{G}_N\right)\\&=N^k\mathbb{P}\Big(X_n^{(1,n)}=...=X_n^{(k_1,n)}=1,...,\\&\quad\quad\quad X_n^{(k_1+...+k_{l-1}+1,n)}=...=X_n^{(k_1+...+k_l,n)}=l|\mathcal{G}_N\Big),\end{align*}

hence after integrating on the random pedigree graph $\mathcal{G}_N$,
\begin{align*}\mathbb{E}&\left(M_n^{k_1}(1)...M_n^{k_l}(l)\right)\\&=N^k\mathbb{P}(X_n^{(1,n)}=...=X_n^{(k_1,n)}=1,...,\\&\quad\quad\quad\quad\quad X_n^{(k_1+...+k_{l-1}+1,n)}=...=X_n^{(k_1+...+k_l,n)}=l).\end{align*}
After integrating on the graph, the sequence $(X_n^{(1,n)},...,X_n^{(k,n)})_{n\in\mathbb{N}}$ is a Markov chain on the finite space $I^k$. We can think of it as the (non independent) motion of $k$ particles on $I$. This Markov chain is irreducible and aperiodic. Indeed, at each time step, $(X_n^{(1,n)},...,X_n^{(k,n)})_{n\in\mathbb{N}}$ stays in the same state with positive probability (at least when the number of occupied sites is strictly smaller than $N$, \textcolor{black}{which will be the case when $N$ goes to infinity}), which gives aperiodicity. Next, starting from any element of $I^k$, it can reach the state $(1,1,...,1)$ with positive probability, for example due to the following sequence of events. During $k$ consecutive time steps, site $1$ is chosen as father, and one of the sites different from $1$ and occupied by $\{X_n^{(1,n)},...,X_n^{(k,n)}\}$ is chosen as child. Then all particles at the child position are sent to the father site $1$. Similarly, starting from the state $(1,1,...,1)$, the Markov chain $(X_n^{(1,n)},...,X_n^{(k,n)})_{n\in\mathbb{N}}$ can reach any state $(i_1,...,i_k)$ of $I^k$ \textcolor{black}{($i_u$ is the position of the $u-th$ particle)} with positive probability, for example due to the following sequence of events. Let us assume without loss of generality that $i_1\neq 1$, and choose $\delta_1\neq \delta_2\in I\setminus \{i_1,...,i_k\}$. Then the Markov chain $(X_n^{(1,n)},...,X_n^{(k,n)})_{n\in\mathbb{N}}$ jumps from $(1,1,...,1)$ to $(i_1\mathbf{1}_{i_u=i_1}+\delta_1\mathbf{1}_{i_u\neq i_1})_{1\leq \textcolor{black}{u}\leq k}$ with positive probability (site $1$ is chosen as child, and sites $i_1$ and $\delta_1$ are chosen as parents; all particles $u$ such that $i_u=i_1$ are sent to parental site $i_1$ and other particles are sent to parental site $\delta_1$). \textcolor{black}{At this stage, particles are either in site $i_1$ or in site $\delta_1$.} Next, the Markov chain $(X_n^{(1,n)},...,X_n^{(k,n)})_{n\in\mathbb{N}}$ jumps from $(i_1\mathbf{1}_{i_u=i_1}+\delta_1\mathbf{1}_{i_u\neq i_1})_{1\leq \textcolor{black}{u}\leq k}$ to $(i_1\mathbf{1}_{i_u=i_1}+i_2\mathbf{1}_{i_u=i_2}+\delta_2\mathbf{1}_{i_u\neq i_1})_{1\leq \textcolor{black}{u}\leq k}$ with positive probability (site $\delta_1$ is chosen as child, and sites $i_2$ and $\delta_2$ are chosen as parents; all particles $u$ such that $i_u=i_2$ are sent to parental site $i_2$ and other particles are sent to parental site $\delta_2$), and so on, alternating $\delta_1$ and $\delta_2$. This gives a path with positive probability and less than $k$ steps leading from $(1,1,...,1)$ to any $(i_1,...,i_k)$.

The Markov chain $(X_n^{(1,n)},...,X_n^{(k,n)})_{n\in\mathbb{N}}$ is then aperiodic and irreducible, and therefore its law converges, when $n$ goes to infinity, to a stationary law on $I^k$, denoted $\nu_{N,k}$. Then for each $k_1, k_2,...,k_l\in\mathbb{Z}_+$ such that $\sum_{j=1}^lk_j=k$,

$$\mathbb{E}\left(M_{\infty}^{k_1}(1)...M_{\infty}^{k_l}(l)\right)=N^k\nu_{N,k}(1,...,1,2,...,2,...,l,...,l)$$ where in the right hand side term the number $i\in[\![1,l]\!]=\{1,2,...,l\}$ is repeated $k_i$ times. Our aim  is then to prove that $$N^k\nu_{N,k}(1,...,1,2,...,2,...,l,...,l)\xrightarrow[N\rightarrow+\infty]{}\prod_{i=1}^l 2^{k_i-1}k_i!,$$ which will give Equation \eqref{eq:limitlawM} in Theorem \ref{mainthm}.

\bigskip
Note that $\nu_{N,k}$ is invariant under any permutation of its $k$ entries and any permutation on $I$, therefore it is determined by  its value on vectors of the form $(1,...,1,2,...,2,...,l,...,l)$. This invariance leads naturally to the following change in state space: for any $x\in I^k$, let us define and denote the configuration associated to $x$ by the multiset $\{x\}=\{k_1,k_2,...,k_l\}$, where $k_1,...,k_l$ are the number of repetitions of each element of $I$ present in $x$. The number $l$, also denoted $L(\{x\})$ will be called the size of the configuration $\{x\}$ (which will sometimes be called an $l$-configuration). As an example, if $N\geq4$, $k=4$ and $x=(3,1,4,4)$ then $\{x\}=\{1,1,2\}$ which has size $3$, and if $N\geq2$, $k=4$ and $x=(1,1,2,1)$, then $\{x\}=\{1,3\}$ which has size $2$.

As mentioned previously, by exchangeability of all sites of $I$ and all entries of $\nu_{N,k}$, if $\{x\}=\{y\}$ then $\nu_{N,k}(x)=\nu_{N,k}(y)$. Now if $\{x\}=\{k_1,...,k_l\}$, then the number of states $y\in I^k$ such that $\{y\}=\{x\}$ is equal to $\textcolor{black}{r(\{x\}):=}\frac{N!}{(N-l)!}\times\frac{k!}{\prod_{i=1}^lk_i!}\textcolor{black}{\times\frac{1}{\prod_{j=1}^kn_j!}}$ \textcolor{black}{in which $n_j$ is the number of i's such that $k_i=j$, i.e. the number of sites occupied by $j$ particles}. The term $\frac{N!}{(N-l)!}$ corresponds to the choice of the $l$ occupied sites, the term $\frac{k!}{\prod_{i=1}^lk_i!}$ to the dispatching of the $k$ Markov chains to these $l$ sites, \textcolor{black}{and the term $\frac{1}{\prod_{j=1}^kn_j!}$ to the exchangeability of sites with same multiplicity}.

\medskip
 So if we denote by $$\nu_{N,k}(\{x\})=\sum_{y\in I^k,\,\{y\}=\{x\}} \nu_{N,k}(y)$$ the limiting (when $n$ goes to infinity) probability  that the Markov chain $(X_n^{(1,n)},...,X_n^{(k,n)})_{n\in\mathbb{N}}$ is in configuration $\{x\}=\{k_1,...,k_l\}$, then we have $$\nu_{N,k}(y)=\frac{\nu_{N,k}(\{x\})}{\textcolor{black}{r(\{x\})}} \quad\text{for each $y\in I^k$ such that $\{y\}=\{x\}$.}$$
So when $N$ goes to infinity, \begin{equation}\label{eq:XtoY}\nu_{N,k}(x)\sim\frac{\nu_{N,k}(\{x\})}{N^{\textcolor{black}{l}}}\frac{\prod_{i=1}^lk_i!}{k!}\textcolor{black}{\times\prod_{j=1}^kn_j!},\end{equation}
and we now study $\nu_{N,k}(\{x\})$.

\bigskip
Let us consider the projection $(Y_n^{(k)})_{n\in\mathbb{N}}$ of the Markov chain $(X_n(1),...,X_n(k))_{n\in\mathbb{N}}$ on the space of configurations of $k$ elements, i.e. on the space $S_k=\{\{k_1,...,k_l\}:k_i\in\mathbb{Z}_+,\sum_{i=1}^lk_i=k\}$. Thanks to the symmetries of the construction, this projection $(Y_n^{(k)})_{n\in\mathbb{N}}$ is in fact an irreducible Markov chain whose transition probabilities are studied now, assuming that $N>k$.

At each time step the Markov chain $(Y_n^{(k)})_{n\in\mathbb{N}}$ can, starting from $\{k_1,...,k_l\}$, either stay at the same point, or jump to a different configuration, which has size $l-1$, $l$, or $l+1$. By considering \textcolor{black}{all} possible events occurring to the population, one obtains (details are given in Appendix \ref{sec:transitionproba}) that for any given positive integers $k_1,k_2,...,k_l$ such that $\sum_{i=1}^l k_i=k$ and any positive integers $k'_1,k'_2,...,k'_{l+1}$, when the population size $N$ goes to infinity,

\begin{equation}\label{eq:transitionproba}
\begin{aligned}\mathbb{P}(Y_{n+1}^{(k)}&=\{k'_1,k'_2,...,k'_{l+1}\}|Y_{n}^{(k)}=\{k_1,k_2,...,k_{l}\})\sim\frac{C(\{k_1,...,k_l\},\{k'_1,...,k'_{l+1}\})}{N}\\&\hspace{5cm} \text{ or $=0$ for all $N$},\\
\mathbb{P}(Y_{n+1}^{(k)}&=\{k'_1,k'_2,...,k'_{l-1}\}|Y_{n}^{(k)}=\{k_1,k_2,...,k_{l}\})\\&\sim \frac{2 C_1(\{k_1,...,k_l\},\{k'_1,...,k'_{l-1}\})}{N^2}, \frac{C_2(\{k_1,...,k_l\},\{k'_1,...,k'_{l-1}\})}{N^3} \\&\hspace{5cm}\text{ or $=0$ for all $N$},\\
\mathbb{P}(Y_{n+1}^{(k)}&=\{k'_1,k'_2,...,k'_{l}\}\neq \{k_1,k_2,...,k_{l}\}|Y_{n}^{(k)}=\{k_1,k_2,...,k_{l}\})\\&\quad\quad\quad\sim\frac{C(\{k_1,...,k_l\},\{k'_1,...,k'_{l}\}\textcolor{black}{)}}{N^2}, \text{ or $=0$ for all $N$},\\
\mathbb{P}(Y_{n+1}^{(k)}&=Y_{n}^{(k)}|Y_{n}^{(k)}=\{k_1,k_2,...,k_{l}\})\\&\quad\quad\quad=\left\{\begin{tabular}{l}
$1-\frac{C(\{k_1,...,k_l\})}{N}+O\left(\frac{1}{N^2}\right)$, \quad if $l< k$, \\ 
$1-\frac{k(k-1)}{N^2}+O\left(\frac{1}{N^3}\right)$, \quad if $l=k$. \\ 
\end{tabular}\right.\end{aligned}\end{equation}
where $C(\{k_1,...,k_l\})$, $C(\{k_1,...,k_l\},\{k'_1,...,k'_{l+1}\})$, $C_1(\{k_1,...,k_l\},\{k'_1,...,k'_{l-1}\})$, $C_2(\{k_1,...,k_l\},\{k'_1,...,k'_{l-1}\})$ and $C(\{k_1,...,k_l\},\{k'_1,...,k'_{l}\})$ are strictly positive constants. 

\medskip
Our proof now relies on the characterization of Markov chain stationary distributions given by the Markov chain tree Theorem, presented notably in \textcolor{black}{\cite{Shubert1975},  \cite{FreidlinWentzell} (Lemma 3.1 in Chapter 6), and \cite{LyonsPeres2016} (Section 4.4.)}. More precisely, for each configuration $\{x\}\in S_k$ let us introduce the set $G(\{x\})$ of  $\{x\}$-graphs which is the set of oriented trees rooted in (and directed to) $\{x\}$, included in the transition graph of $Y^{(k)}$ and spanning all points of $S_k$. For each oriented tree $g\in G(\{x\})$ we define its weight $\pi(g)$ as the product of the transition probabilities of all arrows of $g$, for the Markov chain $Y^{(k)}$.

Then from the Markov chain tree theorem, the stationary distribution of the Markov chain $Y^{(k)}$ is such that for each $\{x\}\in S_k$,

\begin{equation}\label{eq:Freidlin}\nu_{N,k}(\{x\})=\frac{\sum_{g\in G(\{x\})}\pi(g)}{\sum_{\{x\}\in S_k}\sum_{g\in G(\{x\})}\pi(g)}.\end{equation}


%

\bigskip
We now study the trees $g\in G(\{x\})$ and their weights' equivalents when the population size $N$ goes to infinity. Let us fix any configuration $\{x\}=\{k_1,...,k_l\}\in S_k$ and construct an oriented tree $g\in G(\{x\})$, as follows. First, one can find in the transition graph of the Markov chain $(Y_n^{(k)})_{n\in\mathbb{N}}$ a directed path from $\{1,...,1\}$ to $\{x\}$ with exactly $k-l$ steps, for instance by removing one entry and adding it to another entry at each step (as an example if $k=4$ and $\{x\}=\{2,2\}$, the path $\{1,1,1,1\}$-$\{2,1,1\}$-$\{2,2\}$ has positive probability). From \textcolor{black}{the proof of Equations \eqref{eq:transitionproba} in Appendix \ref{appendix},} \textcolor{black}{one can always find such a path whose probability} is equivalent to $C/N^{2(k-l)}$ when $N$ goes to infinity, where $C$ depends on the path but not on $N$. Now for each configuration $\{x'\}=\{k'_1,k'_2,...k'_l\}\neq \{1,...,1\}$ that is not in this path, we choose another configuration $\{x''\}=\{k''_1,k''_2,...k''_{l+1}\}$ such that an arrow from $\{x'\}$ to $\{x''\}$ exists in the transition graph of $(Y^{(k)}_n)_{n\in\mathbb{N}}$. Such a configuration $\{x''\}$ exists, since, assuming without loss of generality that $k\textcolor{black}{'}_1>1$, $(Y^{(k)}_n)_{n\in\mathbb{N}}$ jumps from $\{k'_1,k'_2,...k'_l\}$ to $\{k'_1-1,k'_2,...,k'_l,1\}$ with positive probability. From Equation \eqref{eq:transitionproba}, the transition probability from $\{x'\}$ to $\{x''\}$ is equivalent to $C/N$ when $N$ goes to infinity, where $C$ does not depend on $N$. The concatenation of the path from $\{1,...,1\}$ to $\{x\}$ and these arrows starting outside of this path is a set of $\#S_k-1$ arrows forming \textcolor{black}{a specific} oriented tree $\mathcal{T}$ rooted and directed to $\{x\}$. When the population size $N$ goes to infinity, the weight of this oriented tree is equivalent to
$$\frac{C(\mathcal{T})}{N^{2(k-l)}}\frac{1}{N^{\#S_k-1-(k-l)}}=\frac{C(\mathcal{T})}{N^{\#S_k-1+(k-l)}},$$ where the quantity $C(\mathcal{T})$ does not depend on $N$.  This is also the highest order of magnitude for the weight for an oriented tree pointing \textcolor{black}{to} $\{x\}$ since it contains only $(k-l)$ arrows from an $l'$-configuration to an $l'-1$-configuration which is the minimum number of such transitions, from Appendix \ref{appendix}.

Therefore for each configuration $\{x\}\in S_k$, since $G(\{x\})$ does not depend on $N$, \begin{equation}\label{eq:equivtree}\sum_{g\in G(\{x\})}\pi(g)\sim\frac{C}{N^{\#S_k-1+(k-L(\{x\}))}}.\end{equation}
when $N$ goes to infinity, where $C$ is a positive constant that depends only on $k$.
The term of highest order of the first sum in the denominator of \eqref{eq:Freidlin} is then for the configuration $\{x\}=\{1,1,...,1\}$ for which $L(\{x\})=k$, therefore Equations \eqref{eq:Freidlin} and \eqref{eq:equivtree} yield that
\begin{equation}\label{eq:equivalentnu}\nu_{N,k}(\{x\})\sim C(\{x\})\frac{N^{L(\{x\})}}{N^k}\quad\quad\text{when $N$ goes to infinity.}\end{equation}

Let us now come back to the Markov chain $(X_n^{(1,n)},...,X_n^{(k,n)})_{n\in\mathbb{N}}$, i.e. the annealed distribution of the $k$ particles Markov chain, and denote by $Q^{(N,k)}$ its transition matrix. We know from Equations \eqref{eq:XtoY} and \eqref{eq:equivalentnu} that its stationary law satisfies that for all $x\in I^k$: \begin{equation}\label{eq:nuX}\nu_{N,k}(x)\sim\frac{K(\{x\})}{N^k}\quad\quad\text{when $N$ goes to infinity},\end{equation} where $K(\{x\})=C(\{x\})\frac{\prod_{i=1}^l k_i!}{k!}\textcolor{black}{\times\prod_{j=1}^kn_j!}$ does not depend on $N$. 

\bigskip
We will now prove that $K(\{x\})=\prod_{i=1}^l2^{k_i-1}k_i!$.

The stationary law  $\nu_{N,k}$ of $(X_n^{(1,n)},...,X_n^{(k,n)})_{n\in\mathbb{N}}$ is the unique probability solution of $$\nu_{N,k}=\nu_{N,k}Q^{(N,k)}.$$ This equation can be decomposed as follows, as explained below :
\begin{equation}\label{eq:statnu}
\begin{aligned}\displaystyle
&\nu_{N,k}(1,...,1,2,...,2,...,l,...,l)=\nu_{N,k}(1,...,1,2,...,2,...,l,...,l)\times\frac{N-l}{N}\\&\!+2\sum_{\kappa=l+1}^N\sum_{\substack{\mu\in[\![1,l]\!]\\ \pi\in[\![1,l]\!]\\\mu<\pi}} \frac{1}{N(N-1)(N-2)}\\&\quad\times\sum_{i=0}^{k_{\mu}}\sum_{\substack{j=0 \\ i+j\neq0}}^{k_{\pi}}\left(\frac{1}{2}\right)^{i+j}\sum_{C_2}\nu_{N,k}(1,...,1,...,(\mu,\kappa),...,(\mu,\kappa),...,(\pi,\kappa),...,(\pi,\kappa),...,l,...,l)\\&\!+2\!\!\!\sum_{\kappa=l+1}^N\sum_{\mu=1}^l\sum_{\substack{\pi=l+1\\ \pi\neq\kappa}}^N\!\!\frac{1}{N(N-1)(N-2)}\sum_{i=1}^{\kappa_{\mu}}\left(\frac{1}{2}\right)^i\sum_{C_1}\nu_{N,k}(1,...,1,...,(\mu,\kappa),...,(\mu,\kappa),...,l,...,l).
\end{aligned}
\end{equation}

Note first that transitions under $Q^{(N,k)}$ occur as follows : a) birth and death and parental sites are chosen uniformly and independently among distinct sites (i.e. with probability $1/(N(N-1)(N-2))$ for each choice of a triple) b) particles initially at a birth and death site \textcolor{black}{are} independently assigned to one of the parental sites with probability $1/2$, c) other particles do not change site.

Note that choosing a birth and death position among $\{1,..,l\}$ yields a $0$ transition probability to the state $(1,...,1,2,...,2,...,l,...,l)$. The first term in Equation \eqref{eq:statnu} corresponds to the case where the birth and death position $\kappa$ was chosen among non occupied sites at the previous step (in which case the Markov chain $(X_n^{(1,n)},...,X_n^{(k,n)})$ does not change state and therefore was already in the site $(1,...,1,2,...,2,...,l,...,l)$ of interest). The second term corresponds to the case where $\kappa$ was an initially (i.e. before the transition) occupied site (different from $1$, ...,$l$), and both parental positions $\mu$ and $\pi$ belong to $\{1,...,l\}$. The sum over $C_2$  is a sum over all possible choices of $i$ particles in the position $\mu$ among the $k_{\mu}$ and $j$ particles in the position $\pi$ among the $k_{\pi}$, that were initially in the birth and death position $\kappa$, and the notation $(\mu,\kappa)$ (resp. $(\pi,\kappa)$) means either $\mu$ (resp. $\pi$) or $\kappa$, depending on this choice. The term $\left(1/2\right)^{i+j}$ is the probability that the $i$ chosen particles go to the maternal site $\mu$ while the $j$ others go to the paternal site $\pi$. The third term corresponds to  the case where $\kappa$ was an occupied site (different from $1$, ...,$l$), the mother was chosen among $\{1,...,l\}$ and the father among $\{l+1,...,N\}$ (or conversely, which leads to the "$2$" factor). As before the sum over $C_1$ is a sum over all possible choices of $i$ particles in the position $\mu$ among the $k_{\mu}$, that were at $\kappa$ at the previous step, and the notation $(\mu,\kappa)$ means either $\mu$ or $\kappa$, depending on this choice. The term $\left(1/2\right)^{i}$ is the probability that the $i$ chosen particles all go to the maternal site.

\bigskip
\textcolor{black}{We now use Equation \eqref{eq:nuX} and consider the first order in Equation \eqref{eq:statnu}. Note first that a cancellation occurs for the $0$-order term in $N$, so  the first order in Equation \eqref{eq:statnu} is $1/N$ after this cancellation. Note also that the second term in the right hand-side of \eqref{eq:nuX} does not contribute to this order. Finally, in the last term, in the particular case where $i=\kappa_{\mu}$, $\nu_{N,k}(1,...,1,...,(\mu,\kappa),...,(\mu,\kappa),...,l,...,l)=\nu_{N,k}(1,...,1,2,...,2,...,l,...,l)$.}

\begin{equation}\label{eq:equationK}\begin{aligned}K(\{k_1,...,k_l\})&\times\left[l-2\sum_{\mu=1}^l\left(\frac{1}{2}\right)^{k_\mu}\right]\\&=2\sum_{\mu=1}^l\sum_{i=1}^{k_{\mu}-1}\left(\frac{1}{2}\right)^i\left(\!\begin{array}{c}k_{\mu}\\i\end{array}\right)K(\{k_1,...,k_{\mu}-i,...,k_l,i\}).\end{aligned}\end{equation}

The left term $l-2\sum_{\mu=1}^l\left(\frac{1}{2}\right)^{k_\mu}$ is non null as long as  $\{k_1,...,k_l\}\neq\{1,...,1\}$ so Equation \eqref{eq:equationK} gives the value of $K$ for any $l$-configuration as a function of the values of $K$ for $l+1$-configurations, so Equation \eqref{eq:equationK} admits only one solution once $K(\{1, ..., 1\})$ is fixed, by induction. Therefore all solutions of Equation \eqref{eq:equationK} are proportional. We now prove that for any constant $C(k)$, $K(\{x\})=C(k)\times\prod_{i=1}^l2^{k_i-1}k_i!$ is a solution to this equation. On the right side we get

\begin{align*}C(k)\times 2\sum_{\mu=1}^l\sum_{i=1}^{k_{\mu}-1}\left(\frac{1}{2}\right)^{i}\frac{k_{\mu}!}{i!\,(k_{\mu}-i)!}&\left(\prod_{j=1,j\neq\mu}^lk_j!\,2^{k_j-1}\right)i!\,2^{i-1}(k_{\mu}-i)!\,2^{k_{\mu}-i-1}\\&
=C(k)\times 2\sum_{\mu=1}^l\sum_{i=1}^{k_{\mu}-1}\left(\frac{1}{2}\right)^{i+1}\left(\prod_{j=1}^lk_j!\,2^{k_j-1}\right)\\&
=C(k)\times \left(\prod_{j=1}^lk_j!\,2^{k_j-1}\right)\sum_{\mu=1}^l\frac{\frac{1}{2}-\left(\frac{1}{2}\right)^{k_{\mu}}}{1-\frac{1}{2}}\\&
=C(k)\times \left(\prod_{j=1}^lk_j!\,2^{k_j-1}\right)\left[\sum_{\mu=1}^l\left(1-\left(\frac{1}{2}\right)^{k_{\mu}-1}\right)\right]\end{align*} which is equal to the left-side term of Equation \eqref{eq:equationK}.

\bigskip
We finally prove that $C(k)=1$. Note that $$\sum_{x\in I^k}\nu_{N,k}(x)=\sum_{l=1}^k\sum_{x\in I^k, \textcolor{black}{L(\{x\})}=l}\nu_{N,k}(x)=1.$$ Now for any $l\in\{1,...,k\}$, from Equation \eqref{eq:equivalentnu}, $$\sum_{x\in I^k, \textcolor{black}{L(\{x\})}=l}\nu_{N,k}(x)=\sum_{\{x\}\in S_k, \textcolor{black}{L(\{x\})}=l}\nu_{N,k}(\{x\})\sim \frac{c(k,l)}{N^{k-l}},$$ when population size $N$ goes to infinity, where $c(k,l)$ is a quantity that does not depend on $N$. Therefore when $N$ goes to infinity, $$1=\sum_{x\in I^k}\nu_{N,k}(x)\sim \sum_{x\in I^k, \textcolor{black}{L(\{x\})}=k}\nu_{N,k}(x)\sim K(\{1,...,1\})=C(k).$$
\qed\end{proof}

\bigskip
\begin{proof}[of Corollary \ref{maincor}]
First, if $X$ is a random variable with probability density $\frac{1}{2}\delta_0(t)+\frac{1}{4}e^{-\frac{1}{2}t}$ on $\mathbb{R}_+$, then $$\alpha_k=\mathbb{E}(X^k)=2^{k-1}k!$$ and the power series $\sum_k\alpha_kr^k/k!$ has a positive radius of convergence. From  Theorem \ref{mainthm}, $\mathbb{E}(M_{\infty}(1)^k)\rightarrow \alpha_k$ for all $k\in\mathbb{N}$ when $N\rightarrow\infty$. Then from Theorems 30.1 and 30.2 of \cite{billingsley},  $M_{\infty}(1)\xRightarrow[N\rightarrow\infty]{} X$, and the limiting independence between the random variables $M_{\infty}(1)$,..., $M_{\infty}(l)$ follows similarly, from the product decomposition in the right-hand side of Equation \eqref{eq:limitlawM}.
\qed\end{proof}

\begin{proof}[of Corollary 2]
 \textcolor{black}{The function $\psi:u\mapsto-\mathbf{1}_{u>1/2}\times 2 \ln(2(1-u))$ defined on $[0,1)$ is the inverse of the cumulative distribution function of the random variable $R_1$, which satisfies $\psi^{-1}(x)=1-\frac{1}{2}e^{-x/2}$ for all $x\in\mathbb{R}$. Now let $x\in\mathbb{R}$ and $D_{\infty}(x)=\frac{\#\{j: M_{\infty}(j)\leq x\}}{N}$. We know from Corollary \ref{maincor} that $\mathbb{E}(D_{\infty}(x))=\mathbb{P}(M_{\infty}(1)\leq x)\longrightarrow \mathbb{P}(R_1\leq x)$ when $N$ goes to infinity. Now \begin{align*}\mathbb{E}(D_{\infty}(x)^2)&=\frac{1}{N^2}\mathbb{E}(\#\{j,k: M_{\infty}(j)\leq x,M_{\infty}(k)\leq x\})\\&=\frac{1}{N}\mathbb{P}(M_{\infty}(1)<x)+\frac{N(N-1)}{N^2}\mathbb{P}(M_{\infty}(1)\leq x,M_{\infty}(2)\leq x))\\&\longrightarrow\mathbb{P}(R_1\leq x)^2\quad  \text{when $N$ goes to infinity, from Theorem \ref{mainthm}.}\end{align*} Therefore, for all $x\in\mathbb{R}$, $D_{\infty}(x)$ converges to $\mathbb{P}(R_1\leq x)=\psi^{-1}(x)$ in $L^2$, hence in probability. Note that $$\tilde{M}_j>a\Rightarrow N D_{\infty}(a)\leq j-1$$ and $$\tilde{M}_j<a\Rightarrow ND_{\infty}(a)\geq j.$$ Therefore for any $u>0$,
  \begin{align*}\mathbb{P}(|\tilde{M}_{[Nu]}-\psi(u)|>\epsilon)&\leq\mathbb{P}\Big(D_{\infty}(\psi(u)+\epsilon)\leq \frac{[Nu]-1}{N}\Big)+\mathbb{P}\Big(D_{\infty}(\psi(u)-\epsilon)\geq \frac{[Nu]}{N}\Big)\\&\leq \mathbb{P}(D_{\infty}(\psi(u)+\epsilon)\leq u-\frac{1}{N})+\mathbb{P}(D_{\infty}(\psi(u)-\epsilon)\geq u-\frac{1}{N})\\&=\mathbb{P}\Big(D_{\infty}(\psi(u)+\epsilon)\leq \psi^{-1}(\psi(u)+\epsilon)\\&\quad\quad\quad\quad+[u-\frac{1}{N}-\psi^{-1}(\psi(u)+\epsilon)]\Big)\\&+\mathbb{P}\Big(D_{\infty}(\psi(u)-\epsilon)\geq \psi^{-1}(\psi(u)-\epsilon)\\&\quad\quad\quad\quad+[u-\frac{1}{N}-\psi^{-1}(\psi(u)-\epsilon)]\Big).\end{align*} Now $\psi^{-1}$ is strictly increasing on $\mathbb{R}_+$ so $\psi^{-1}(\psi(u)+\epsilon)-u=C_+(u,\epsilon)>0$ for all $u>0$, and since $\psi^{-1}(v)=0$ if $v<0$, $u-\psi^{-1}(\psi(u)-\epsilon)=C_-(u,\epsilon)>0$ for all $u>0$. So \begin{align*}\mathbb{P}(|\tilde{M}_{[Nu]}-\psi(u)|>\epsilon)&\leq\mathbb{P}(D_{\infty}(\psi(u)+\epsilon)\leq \psi^{-1}(\psi(u)+\epsilon)-C_+(u,\epsilon)-\frac{1}{N})\\&+\mathbb{P}\Big(D_{\infty}(\psi(u)-\epsilon)\geq \psi^{-1}(\psi(u)-\epsilon)-\frac{1}{N}+C_-(u,\epsilon)\Big)\\&\leq \mathbb{P}(D_{\infty}(\psi(u)+\epsilon)\leq \psi^{-1}(\psi(u)+\epsilon)-C_+(u,\epsilon))\\&+\mathbb{P}\Big(D_{\infty}(\psi(u)-\epsilon)\geq \psi^{-1}(\psi(u)-\epsilon)+\frac{C_-(u,\epsilon))}{2} \Big)\end{align*} if $N$ is large enough. So $\mathbb{P}(|\tilde{M}_{[Nu]}-\psi(u)|\geq\epsilon)$ goes to $0$ when $N$ goes to infinity.
} 
 
  \qed
\end{proof}


\textcolor{black}{\textbf{Acknowledgement: } We are very thankful to the referees as well as the associate editor for their useful comments and criticisms which helped us considerably improve the manuscript.}

\appendix

\section{Transition matrix of the configuration Markov chain}\label{appendix}

In this section we derive the transition probabilities of the Markov chain $\left(Y_n^{(k)}\right)_{n\in\mathbb{N}}$, whose states are configurations of the form $\{k_1,k_2,...,k_l\}$ where $k_i\in\mathbb{Z}_+$, $\sum_{i=1}^l k_i=k$ and $l$ is called the size of the configuration. We distinguish $4$ types of events occurring to this chain: jumping from a configuration with size $l$ to a configuration with size $l+1$, jumping from a configuration with size $l$ to a configuration with size $l-1$, jumping from a configuration with size $l$ to a different configuration with size $l$, and staying in the same configuration.

\bigskip
The size of the configuration is increased by $1$ during one time step if one of the $l$ occupied sites is chosen as position of the child and all the Markov chains present at that site are sent on exactly two parental positions that are distinct from the already occupied sites (both new parental positions must be occupied after the repartition of the Markov chains present at the children site). This gives that the probability to jump from $\{k_1,...,k_l\}$ to a given state $\{k'_1,...,k'_{l+1}\}$ (supposing that $N>k$) is equal to

\begin{align}\label{eq:l->l+1}\frac{(N-l)(N-l-1)}{(N-1)(N-2)}\times \frac{C(\{k_1,...,k_l\},\{k'_1,...,k'_{l+1}\})}{N}\hspace{1cm} \text{or $0$},\end{align} where the quantity \begin{align*}C(\{k_1,...,k_l\},\{k'_1,...,k'_{l+1}\})&=\sum_{u=1}^l\sum_{a=0}^{k_{u}-1}\left(\begin{array}{c}
k_u \\ 
a
\end{array} \right)\left(\frac{1}{2}\right)^{k_u}\\&\times\mathbf{1}_{\{k'_1,...,k'_{l+1}\}=\{k_1,...,k_{u-1},k_{u+1},...,k_l, a,k_u-a\}}\end{align*} does not depend on $N$, which gives the equivalence result in the first line of Equation \eqref{eq:transitionproba}. Note that in the particular case where $l=k$ then the transition probability from $\{k_1,...,k_l\}$ to any state $\{k'_1,...,k'_{l+1}\}$ is equal to $0$.

\bigskip
The size of the configuration is decreased by $1$ if one of the occupied sites is chosen as child position and all the Markov chains present at that site are sent on one or two already occupied sites (which can happen either when both parental positions were already occupied, or when one parental position was not already occupied but no Markov chains present at the child position are sent to this parent). This gives that the probability for the Markov chain $\left(Y_n^{(k)}\right)_{n\in\mathbb{N}}$ to jump from $\{k_1,...,k_l\}$ to a given state $\{k'_1,...,k'_{l-1}\}$ (supposing that $N>k$) is equal to

\begin{equation}
\begin{aligned}\label{eq:l->l-1-general}&\frac{2(N-l)}{N(N-1)(N-2)} C_1(\{k_1,...,k_l\},\{k'_1,...,k'_{l-1}\}))\\&+\frac{1}{N(N-1)(N-2)}C_2(\{k_1,...,k_l\},\{k'_1,...,k'_{l-1}\})\end{aligned}\end{equation} where the quantities \begin{align*}C_1(&\{k_1,...,k_l\},\{k'_1,...,k'_{l-1}\})=\sum_{1\leq u\neq v\leq l}\left(\frac{1}{2}\right)^{k_{u}}\\&\times\mathbf{1}_{\{k'_1,...,k'_{l-1}\}=\{k_1,...,k_{u-1},k_{u+1},...,k_{v-1},k_{v+1},...,k_l, k_{v}+k_u\}}\end{align*} and \begin{align*}C_2(&\{k_1,...,k_l\},\{k'_1,...,k'_{l-1}\})=\sum_{1\leq u\neq v\neq w\leq l}\,\sum_{a=0}^{k_{u}}\left(\begin{array}{c}
k_{u} \\ 
a
\end{array} \right)\left(\frac{1}{2}\right)^{k_{u}}\\&\times\mathbf{1}_{\{k'_1,...,k'_{l-1}\}=\{k_1,...,k_{u-1},k_{u+1},...,k_{v-1},k_{v+1},...,k_{w-1},k_{w+1},...,k_l, k_{w}+a,k_{v}+k_{u}-a\}},\end{align*} do not  depend on $N$. Note that if there exist $u\neq v$ in $[\![1,l]\!]$ such that $k'_{v}=k_{v}+k_{u}$ and $k'_i=k_i$ otherwise, then the quantity $C_2(\{k_1,...,k_l\},\{k'_1,...,k'_{l-1}\})$ is positive. This gives the equivalence result stated in the second line of Equation \eqref{eq:transitionproba}.

\bigskip
To keep exactly $l$ occupied sites while changing state, we need to choose one parental site among the occupied sites and one parental site among the unoccupied sites, and at least one Markov chain present at the child site must choose the last one.  This gives that the probability for the Markov chain $\left(Y_n^{(k)}\right)_{n\in\mathbb{N}}$ to jump from $\{k_1,...,k_l\}$ to a given state \textcolor{black}{$\{k'_1,...,k'_{l}\}\neq \{k_1,...,k_l\}$} is equal, when $N>k$, to

\begin{equation}\label{eq:l->l}\frac{2(N-l)}{N(N-1)(N-2)}C(\{k_1,...,k_l\},\{k'_1,...,k'_l\}),\end{equation}
where the quantity \begin{align*}C(\{k_1,...,k_l\},&\{k'_1,...,k'_{l}\})=\sum_{1\leq u\neq v\leq l} \,\sum_{a=1}^{k_{u}}\left(\begin{array}{c}
k_{u} \\ 
a
\end{array} \right)\left(\frac{1}{2}\right)^{k_{u}}\\&\times\mathbf{1}_{\{k'_1,...,k'_{l}\}=\{k_1,...,k_{u-1},k_{u+1},...,k_{v-1},k_{v+1},...,k_l, a,k_{v}+k_{u}-a\}}\end{align*} is independent of $N$, which gives the equivalence result in the third line of Equation \eqref{eq:transitionproba}.

\bigskip
The last event is when the Markov chain $\left(Y_n^{(k)}\right)_{n\in\mathbb{N}}$ stays on the same state $\{k_1,...,k_l\}$. From previous calculations we get that the probability of this event is equal to 
\begin{align}\label{eq:stay}\left\{\begin{tabular}{l}
$1-\frac{C(\{k_1,...,k_l\})}{N}+O\left(\frac{1}{N^2}\right)\text{if $l\neq k$} $\\
$1-\frac{k(k-1)}{N^2}+O\left(\frac{1}{N^3}\right)\text{if $l=k$.}$
\end{tabular} \right.\end{align}
The first line is directly given by Equations \eqref{eq:l->l+1}, \eqref{eq:l->l-1-general} and \eqref{eq:l->l}. The second line can be calculated directly without using these equations. We indeed consider the case in which all $k$ particles are located in different sites, and consider the probability for the configuration Markov chain to change state during one time step. For this event to occur, one of the occupied sites must be chosen for the child position, and the associated particle must be sent to another of these occupied sites, which gives a $k(k-1)/(N(N-1))$ probability.


\label{sec:transitionproba}
\bibliographystyle{abbrv}
 \bibliography{biblio_genealogy}

\begin{thebibliography}{10}

\bibitem{BairdBartonEtheridge2003}
S.~Baird, N.~Barton, and A.~Etheridge.
\newblock The distribution of surviving blocks of an ancestral genome.
\newblock {\em Theor. Pop. Biol.}, 64(4):451--71, 2003.

\bibitem{BartonEtheridge2011}
N.~H. Barton and A.~M. Etheridge.
\newblock The relation between reproductive value and genetic contribution.
\newblock {\em Genetics}, 188(4):953--973, 2011.

\bibitem{billingsley}
P.~Billingsley.
\newblock {\em Convergence of probability measures}.
\newblock Wiley Series in Probability and Statistics: Probability and
  Statistics. John Wiley \& Sons Inc., second edition, 1999.

\bibitem{BirknerLiuSturm2018}
M.~Birkner, H.~Liu, and A.~Sturm.
\newblock {Coalescent results for diploid exchangeable population models}.
\newblock {\em Electronic Journal of Probability}, 23(none):1 -- 44, 2018.

\bibitem{Birkneretal2013}
E.~B. Birkner~M, Blath~J.
\newblock An ancestral recombination graph for diploid populations with skewed
  offspring distribution.
\newblock {\em Genetics}, 193(1):255--290, 2013.

\bibitem{Chang1999}
J.~T. Chang.
\newblock Recent common ancestors of all present-day individuals.
\newblock {\em Advances in Applied Probability}, 31(4):1002--1026, 1999.

\bibitem{Derrida2000}
B.~Derrida, S.~C. Manrubia, and D.~H. Zanette.
\newblock On the genealogy of a population of biparental individuals.
\newblock {\em Journal of Theoretical Biology}, 203(3):303 -- 315, 2000.

\bibitem{FreidlinWentzell}
M.~Freidlin, J.~Sz{\"u}cs, and A.~Wentzell.
\newblock {\em Random Perturbations of Dynamical Systems}.
\newblock Grundlehren der mathematischen Wissenschaften. Springer, 1984.

\bibitem{GravelSteel2015}
S.~Gravel and M.~Steel.
\newblock The existence and abundance of ghost ancestors in biparental
  populations.
\newblock {\em Theor Popul Biol}, 101:47--53, 2015.

\bibitem{Lambertetal2018}
A.~Lambert, V.~Mir\'o~Pina, and E.~Schertzer.
\newblock Chromosome painting.
\newblock {\em arXiv:1807.09116}, 2018.

\bibitem{Levy1937}
P.~Levy.
\newblock {\em Th\'eorie de l'addition des variables al\'eatoires}.
\newblock Gauthier-Villars, Paris, 1937.

\bibitem{Linder2009CommonAI}
M.~Linder.
\newblock Common ancestors in a generalized {M}oran model.
\newblock {\em U.U.D.M. Reports}, 2009.

\bibitem{LyonsPeres2016}
R.~Lyons and Y.~Peres.
\newblock {\em Probability on Trees and Networks}, volume~42 of {\em Cambridge
  Series in Statistical and Probabilistic Mathematics}.
\newblock Cambridge University Press, New York, 2016.

\bibitem{MatsenEvans2008}
F.~A. Matsen and S.~A. Evans.
\newblock To what extent does genealogical ancestry imply genetic ancestry?
\newblock {\em Theoretical Population Biology}, 2008.

\bibitem{Mohle1998}
M.~M\"{o}hle.
\newblock A convergence theorem for markov chains arising in population
  genetics and the coalescent with selfing.
\newblock {\em Advances in Applied Probability}, 30(2):493--512, 1998.

\bibitem{MohleSagitov2003}
M.~M\"{o}hle and S.~Sagitov.
\newblock Coalescent patterns in diploid exchangeable population models.
\newblock {\em J. Math. Biol.}, 47:337--352, 2003.

\bibitem{Shubert1975}
B.~O. Shubert.
\newblock A flow-graph formula for the stationary distribution of a markov
  chain.
\newblock {\em IEEE Transactions on Systems, Man, and Cybernetics},
  SMC-5(5):565--566, 1975.

\bibitem{Wakeley2005}
J.~Wakeley.
\newblock The limits of theoretical population genetics.
\newblock {\em Genetics}, 169:1--7, 2005.

\bibitem{Wakeley2012}
J.~Wakeley.
\newblock Gene genealogies within a fixed pedigree, and the robustness of
  kingman(s coalescent.
\newblock {\em Genetics}, 190:1433--1445, 2012.

\bibitem{Wakeley2016}
J.~Wakeley.
\newblock Effects of the population pedigree on genetic signatures of
  historical demographic events.
\newblock {\em PNAS}, 113(29):7994--8001, 2016.

\bibitem{Wilton2017}
P.~R. Wilton, P.~Baduel, M.~M. Landon, and J.~Wakeley.
\newblock Population structure and coalescence in pedigrees: Comparisons to the
  structured coalescent and a framework for inference.
\newblock {\em Theoretical Population Biology}, 115:1 -- 12, 2017.

\end{thebibliography}

\end{document}